\documentclass[conference]{IEEEtran}
\usepackage{cite}
\usepackage{amsmath,amssymb,amsfonts}
\usepackage{amsthm}
\usepackage{algorithm}
\usepackage{algpseudocode}
\usepackage{graphicx}
\usepackage{tikz}
\usetikzlibrary{decorations.pathreplacing,angles,quotes}
\usepackage{textcomp}
\usepackage{xcolor}
\def\BibTeX{{\rm B\kern-.05em{\sc i\kern-.025em b}\kern-.08em
    T\kern-.1667em\lower.7ex\hbox{E}\kern-.125emX}}

\newtheorem{theorem}{Theorem}[section]
\newtheorem{lemma}[theorem]{Lemma}
\newtheorem{corollary}[theorem]{Corollary}

\newtheorem{remark}{Remark}[section]

\begin{document}

\title{Superlinear Convergence of Randomized Block Lanczos Algorithm
}

\author{\IEEEauthorblockN{Qiaochu Yuan}
\IEEEauthorblockA{\textit{Department of Mathematics} \\
\textit{UC Berkeley}\\
Berkeley, CA, USA\\
qyuan@berkeley.edu\\
}
\and
\IEEEauthorblockN{Ming Gu}
\IEEEauthorblockA{\textit{Department of Mathematics} \\
\textit{UC Berkeley}\\
Berkeley, CA, USA\\
mgu@math.berkeley.edu\\
}
\and
\IEEEauthorblockN{Bo Li}
\IEEEauthorblockA{\textit{Department of Mathematics} \\
\textit{UC Berkeley}\\
Berkeley, CA, USA\\
bo\_li@berkeley.edu \\
}
}

\maketitle

\begin{abstract}
The low rank approximation of matrices is a crucial component in many data mining applications today. A competitive algorithm for this class of problems is the randomized block Lanczos algorithm - an amalgamation of the traditional block Lanczos algorithm with a randomized starting matrix. While empirically this algorithm performs quite well, there has been scant new theoretical results on its convergence behavior and approximation accuracy, and past results have been restricted to certain parameter settings. In this paper, we present a unified singular value convergence analysis for this algorithm, for all valid choices of the block size parameter. We present novel results on the rate of singular value convergence and show that under certain spectrum regimes, the convergence is superlinear. Additionally, we provide results from numerical experiments that validate our analysis. 
\end{abstract}


\section{Introduction}
The low rank approximation of matrices is a crucial component in many data mining applications today. In addition to functioning as a stand alone technique for dimensionality reduction \cite{cohen2015dimensionality}, denoising \cite{nguyen2013denoising}, signal processing \cite{fazel2008compressed}, data compression \cite{pmlr-v70-anderson17a}, and more, it has also been incorporated into more complex algorithms as a computational subroutine \cite{liu2013tensor,parikh2014proximal}. As part of large scale modern data processing, low rank approximations help to reveal important structural information in the raw data and to transform the data into forms that are more efficient for computation, transmission, and storage. 

The singular value decomposition (SVD) is a matrix factorization of both theoretical and practical importance, and it has a number of useful properties related to matrix nearness and rank. In particular, it is used to identify nearby matrices of lower rank, and, leaving aside the question of computational complexity, it is known that the rank-$k$ truncated SVD is the ``gold standard'' for approximating a matrix by another matrix of rank at most $k$ \cite{eckart1936approximation}.

While procedures for computing the exact rank-$k$ truncated SVD have existed since the 1960s \cite{golub1965calculating}, the computational cost of these algorithms are prohibitive at the scale of many of today's datasets. The recent applications of low rank matrix approximation techniques to big-data problems differ in both the computation efficiency requirement and the accuracy requirement of the algorithms. Firstly, we are increasingly leaving behind the era of moderately sized matrices and entering an age of web-scale datasets and big-data applications. The matrices arising from such are often extraordinarily large, exceeding the order of $10^6$ in one or both of the dimensions \cite{talwalkar2013large,mazumder2010spectral,cohen2012survey}, and have much higher computational efficiency demands on the algorithms. Secondly, while the truncated SVD may be the final desired object for previous scientific computing questions, for big-data applications, it is usually an intermediate representation for the overall classification or regression task. Empirically, the final accuracy of the task only weakly depends on the accuracy of the matrix approximation \cite{gu2015subspace}. Thus, while previous variants of truncated SVD algorithms focused on computing up to full double precision, newer iterations of these algorithms aimed at big-data applications can comfortably get by with only $2$-$3$ digits of accuracy. 

These considerations have led to the development of randomized variants of traditional SVD algorithms suited to large, sparse matrices, in particular randomized subspace iteration (RSI) and randomized block Lanczos (RBL) \cite{drineas2006fast,rokhlin2009randomized,halko2011finding,musco2015randomized}. By applying either a randomized sketching or projecting operation on the original matrix, these algorithms balance reducing computational complexity with producing an acceptably accurate approximation. While empirically they have shown to be effective and have been widely adopted by popular software packages, e.g. \cite{pedregosa2011scikit}, there has been scant new theoretical work on the convergence guarantees of the latter algorithm, the better performing but more complicated randomized block Lanczos algorithm. 

In this paper, we present novel theoretical convergence results concerning the rate of singular value convergence for the RBL algorithm, along with numerical experiments supporting these results. Our analysis presents a unified singular value convergence theory for variants of the Block Lanczos algorithm, for all valid parameter choices of block size $b$. To our knowledge, all previous results in the literature are applicable only for the choice of $b \geq k$, the target rank. We present a generalized theorem, applicable to all block sizes $b$, which coincide asymptotically with previous results for the case $b \geq k$, while providing equally strong rates of convergence for the case $b < k$.  

In Section~\ref{sec:backbround}, we present the randomized block Lanczos algorithm and discuss some previous convergence results for this algorithm. In Section~\ref{sec:theoretical_results}, we dive into our main theoretical result and its derivation, followed by corollaries for special cases. In Section~\ref{sec:numerical_experiments}, we investigate the behavior of this algorithm for different parameter settings and empirically verify the results of the previous section. Finally, we give concluding remarks in Section~\ref{sec:conclusions}.

\section{Background} \label{sec:backbround}

\subsection{Preliminaries}

Throughout this paper, our analysis assumes exact arithmetics. 

We denote matrices by bold-faced uppercase letters, e.g. $\mathbf{M}$, entries of matrices by the plain-faced lowercase letter that the entry belongs to, e.g. $m_{11}$, and block submatrices by the bold-faced or script-faced uppercase letter that the submatrix belongs to subscripted by position, possibly with subscripts, e.g. $\mathbf{M}_{11}$, $\mathcal{M}_{11}$ or $\mathbf{M}_{a \times b}$. Double numerical subscripts denote the position of the element or the submatrix, i.e. $\mathbf{M}_{11}$ and $m_{11}$ are the topmost leftmost subblock or entry of $\mathbf{M}$ respectively. $m \times n$ subscripts denote the dimensions of a submatrix, when such information is relevant, i.e. $\mathbf{M}_{a \times b}$ denote a subblock of $\mathbf{M}$ that has dimensions $a \times b$. 

Constants are denoted by script-faced uppercase or lowercase letters, e.g. $\mathcal{C}$ or $\mathcal{\alpha}$, when it is asymptotically insignificant, i.e. constant with respect to the convergence parameter.

The SVD of a matrix $\mathbf{A}$ is defined as the factorization
\begin{equation}
\mathbf{A} = \mathbf{U} \mathbf{\Sigma} \mathbf{V}^T
\end{equation}
where $\mathbf{U} = \begin{bmatrix} \mathbf{u}_1 & \cdots & \mathbf{u}_n \end{bmatrix}$ and $\mathbf{V} = \begin{bmatrix} \mathbf{v}_1 & \cdots & \mathbf{v}_n \end{bmatrix}$ are orthogonal matrices whose columns are the set of left and right singular vectors respectively, and $\mathbf{\Sigma}$ is a diagonal matrix whose entries $\mathbf{\Sigma}_{ii} = \sigma_i$ are the singular values ordered descendingly $\sigma_1 \geq \cdots \geq \sigma_n \geq 0$.

The rank-$k$ truncated SVD of a matrix is defined as 
\begin{equation}
\mathrm{svd}_k\left( \mathbf{A} \right) = \mathbf{U}_k \mathbf{\Sigma}_k \mathbf{V}_k
\end{equation}
where $\mathbf{U}_k = \begin{bmatrix} \mathbf{u}_1 & \cdots & \mathbf{u}_k \end{bmatrix}$ and $\mathbf{V}_k = \begin{bmatrix} \mathbf{v}_1 & \cdots & \mathbf{v}_k \end{bmatrix}$ contain the first $k$ left and right singular vectors respectively, and $\mathbf{\Sigma}_k = \mathrm{diag}(\sigma_1, \cdots, \sigma_k)$.

The $i$th singular values of an arbitrary matrix $\mathbf{M}$ is denoted by $\sigma_i(\mathbf{M})$, or simply $\sigma_i$ when it is clear from context the matrix in question. 

The $p$th degree Chebyshev polynomial is defined by the recurrence
\begin{align}
T_0(x) &\equiv 1 \\
T_1(x) &\equiv x \\
T_p(x) &\equiv 2p T_{p-1}(x) - T_{p-2}(x)
\end{align}
Alternatively, they may be expressed as
\begin{equation}
T_p(x) = \frac{1}{2} \left( \left(x + \sqrt{x^2 - 1}\right)^{p} + \left(x + \sqrt{x^2 - 1} \right)^{-p} \right)
\end{equation}
for $\vert x \vert > 1$, and estimated as
\begin{equation}
T_p(1+\epsilon) \approx \frac{1}{2} \left( 1 + \epsilon + \sqrt{2 \epsilon} \right)^p
\end{equation}
for $p$ large and $\epsilon$ small. 

\subsection{The Algorithm}

The randomized block Lancos algorithm is a straightforward combination of the classical block Lanczos algorithm \cite{golub1977block} with the added element of a randomized starting matrix $\mathbf{V} = \mathbf{A} \mathbf{\Omega}$. 

The pseudocode for this algorithm is outlined in Algorithm~\ref{alg:blk_lanczos}. Of the parameters of the algorithm, $k$ (target rank) is problem dependent, while $b$ (block size), $q$ (no. of iterations) are chosen by the user to control the quality and computational cost of the approximation. The algorithm requires the choices of $b, q$ to satisfy $qb \geq k$, to ensure that the Krylov subspace be at least $k$ dimensional. 

\begin{algorithm}
	\caption{randomized block Lanczos algorithm pseudocode}
	\label{alg:blk_lanczos}
	\begin{algorithmic}[1]
		\Require 
		$\begin{array}{ll}
		\mathbf{A} \in \mathbb{R}^{m \times n} & \\
		\mathbf{\Omega} \in \mathbb{R}^{n \times b} & \textrm{, random Gaussian  matrix} \\
		k & \textrm{, target rank} \\
		b & \textrm{, block size} \\
		q & \textrm{, number of Lanczos iterations}
		\end{array}$
		\Ensure $\begin{array}{ll}
		\mathbf{B}_k \in \mathbb{R}^{m \times n} & \textrm{, a rank-$k$ approximation to $\mathbf{A}$}
		\end{array}$ 
		\State Form the block column Krylov subspace matrix \hspace{1cm} $\mathbf{K} = \begin{bmatrix} \mathbf{A} \mathbf{\Omega} & (\mathbf{A}\mathbf{A}^T) \mathbf{A} \mathbf{\Omega} & \cdots & (\mathbf{A}\mathbf{A}^T)^{q}\mathbf{A}\mathbf{\Omega} \end{bmatrix}$.
		\State Compute an orthonormal basis $\mathbf{Q}$ for the column span of $\mathbf{K}$, using e.g. $\mathbf{Q}\mathbf{R} \leftarrow \mathrm{qr}(\mathbf{K})$.
		\State Project $\mathbf{A}$ onto the Krylov subspace by computing \hspace{1cm} $\mathbf{B} = \mathbf{Q} \mathbf{Q}^T \mathbf{A}$.
		\State  Compute $k$-truncated SVD $\mathbf{B}_k = \mathrm{svd}_k \left( \mathbf{B} \right) = \mathrm{svd}_k \left( \mathbf{Q}\mathbf{Q}^T \mathbf{A} \right) = \mathbf{Q} \cdot \mathrm{svd}_k \left( \mathbf{Q}^T \mathbf{A} \right)$.
		\State Return $\mathbf{B}_k$.
	\end{algorithmic}
\end{algorithm}

We present the algorithm pseudocode in this form in order to highlight the mathematical ideas that are at the core of this algorithm. It is well known that a naive implementation of any Lanczos algorithm is plagued by loss of orthogonality of the Lanczos vectors due to roundoff errors \cite{paige_article}. A practical implementation of Algorithm~\ref{alg:blk_lanczos} should involve, at the very least, a reorganization of the computation to use the three-term recurrence and bidiagonalization \cite{golub}, and reorthogonalizations of the Lanczos vectors at each step using one of the numerous schemes that has been proposed \cite{golub,parlett,simon}. 

\subsection{Previous Work}

Historically, the the classical Lanczos algorithm was developed as an eigenvalue algorithm for symmetric matrices. Its convergence analysis focused on theorems concerning the approximation quality of the approximant's eigenvalues as a function of $k$, the target rank. The analysis relied heavily on the analysis of the $k$-dimensional Krylov subspace and the choice of the associated $k$-degree Chebyshev polynomial. Classical results in this line of inquiry include those by Kaniel \cite{kaniel1966estimates}, Paige \cite{paige1971computation}, Underwood \cite{underwood1975iterative}, Saad \cite{saad_article}. 

More recently, while there has been much work on the analysis of randomized algorithms, such efforts have been focused mostly on RBL's simpler cousins, such as randomized power iteration or randomized subspace iteration \cite{halko2011finding,gu2015subspace}. The exception is the results from \cite{musco2015randomized}. To our knowledge, this is one of the few works that provide convergence analysis for randomized block Lanczos and the first work that gives ``gap''-independent theoretical bounds for this algorithm. The analysis found therein is restricted to the case for the block size, $b$, chosen at least the size of $k$, the desired target rank. Our theoretical analysis will give a more generally applicable convergence bound, encompassing the case for both $1 \leq b < k$ and $b \geq k$. In the latter case, our theoretical results will coincide with those in \cite{musco2015randomized}. In the former case, we show that the rapid convergence of the algorithm for any block size $b$ larger than the largest singular value cluster size is assured. We draw attention to this distinction in choosing the block size parameter $b$ - in our numerical experiments, we show that generally smaller choices for $b$ are favored. 

Our current work is based partially on the analysis found in \cite{gu2015subspace}. This work established aggressive multiplicative convergence bounds for the randomized subspace iteration algorithm, for both singular values and normed (Frobenius, spectral) matrix convergence. These bounds depend on both the singular value gap and the number of iterations taken by the algorithm - the former is a property of the matrix in question, and the latter is proportional to the computational complexity of the algorithm. The analysis presented in this work is linear algebraic in nature, drawing on deterministic matrix analysis, as well expectation bounds on randomized Gaussian matrices and their concentration of measure characteristics. Our current work employs similar methods, and achieves bounds of a similar form. While the details differ, core ideas, such as creating an artificial ``gap'' in the spectrum and choosing an opportune orthonormal basis for the analysis, are the same.

\section{Theoretical Results} \label{sec:theoretical_results}

\subsection{Problem Statement}
Given an arbitrary matrix $\mathbf{A} \in \mathbb{R}^{m \times n}$ and a target rank $k \leq \mathrm{rank}(\mathbf{A})$, the goal of a low-rank matrix approximation algorithm is to compute another matrix $\mathbf{B}_k \in \mathbb{R}^{m \times n}$ whose rank is at most $k$. 

There are many ways to ask and answer the question, ``how good of an approximation is $\mathbf{B}_k$ to the original $\mathbf{A}$?'' In particular, for various low-rank approximation algorithms, the answer has been provided in terms of normed approximation error \cite{halko2011finding,gu2015subspace,musco2015randomized,xiao2016spectrum}, singular subspace error \cite{chen2009lanczos,li2015convergence}, and singular value error \cite{gu2015subspace,saad_article}. 

In this paper, we focus on the singular value error for the randomized block Lanczos algorithm. As $\mathbf{B}$ is an orthogonal projection of $\mathbf{A}$ in Alg.~\ref{alg:blk_lanczos}, by the Cauchy interlacing theorem for singular values, we immediately have the upper bound
\begin{equation}
\sigma_j \geq \sigma_j(\mathbf{B}_k) 
\end{equation}
for $j = 1, \cdots, k$. 

The optimal lower bound is achieved, of course, by the rank-$k$ truncated SVD of $\mathbf{A}$, giving the tight inequality
\begin{equation}
\sigma_j \geq \sigma_j(\mathrm{svd}_k(\mathbf{A})) \geq \sigma_j
\end{equation}
for $j = 1, \cdots, k$. 

We will to show that the randomized block Lanczos algorithm provides competitive accuracy, and produces singular value estimates at least some fraction of the optimum.
\begin{equation}
\sigma_j \geq \sigma_j(\mathbf{B}_k) \geq \frac{\sigma_j}{\sqrt{1 + \{\text{some convergence factor} \}}}
\end{equation}
for $\{\text{some convergence factor}\} \rightarrow 0$. 

\subsection{Key Results}
Our convergence analysis will show that if the randomized block Lanczos algorithm converges, then the $k$ desired singular values of the approximation $\mathbf{B}_k$ converges to the corresponding true singular values of $\mathbf{A}$ exponentially in the number of iteration $q$. Moreover, convergence occurs as long as the block size $b$ is chosen to be larger than the maximum cluster size for the $k$ relevant singular values.

We present our main results here and delay their proofs to Subsection~\ref{ssec:analysis}. Our main theorem is as follows.
\begin{theorem}
	\label{thm:svconvergence}
	Let $\mathbf{B}_k$ be the matrix returned by Alg.~\ref{alg:blk_lanczos}. Assume that $\mathbf{\Omega}$ is chosen such that the two conditions in Remark~\ref{lst:conditions} hold. For any choices of $r, s \geq 0$, and any parameter choice $b, q$ satisfying $k+r = (q-p)b \geq k$, for $j = 1, \cdots, k$,
	\begin{equation}
	\sigma_j \geq \sigma_j (\mathbf{B}_k) \geq \frac{\sigma_{j+s}}{\sqrt{1 + \mathcal{C}^2 T_{2p+1}^{-2} \left( 1 + 2 \cdot \frac{\sigma_j - \sigma_{j+s+r+1}}{\sigma_{j+s+r+1}} \right) }} \label{eqn:svconvergence_withs}
	\end{equation}
	where $\mathcal{C}$ is a constant that is independent of $q$.
\end{theorem} 

This inequality shows that for all valid choices of parameters $b, q$, the convergence of the approximate singular values are governed by the growth of the Chebyshev polynomial term
\begin{equation}
T_{2p+1} \left( 1 + 2 \cdot \frac{\sigma_j - \sigma_{j+s+r+1}}{\sigma_{j+s+r+1}} \right)
\end{equation}
, with the bound holding across all choices of the analysis parameters $s, r$.

Theorem~\ref{thm:svconvergence} admits the following corollaries about two special choices for the block size parameter $b$, where the constants in each case can be expressed in an algebraically closed form.  
\begin{corollary}[Special case: $b = 1$]
	For any choices of $r, s \geq 0$ satisfying $k+r = (q-p) \geq k$, for $j = 1, \cdots, k$,
	\begin{equation}
	\sigma_j \geq \sigma_j \left(\mathbf{B}_k \right) \geq \frac{\sigma_{j+s}}{\sqrt{1 + \mathcal{C}_{b=1} T_{2p+1}^{-2} \left(1 + 2 \cdot \frac{\sigma_j - \sigma_{j+s+r+1}}{\sigma_{j+s+r+1}} \right)}}
	\end{equation}
	where 
	\begin{equation}
	\mathcal{C}_{b = 1} = \left( \max_{\overset{1 \leq s \leq k}{j+r+1 \leq r \leq n}} \frac{\widehat{\omega}_r}{\widehat{\omega}_s} \right)^2 \cdot \left( \sum_{s=1}^j \sum_{r = j+r+1}^{n} \prod_{\overset{t=1}{t \neq s}}^{j+r}  \left( \frac{\sigma_r^2 - \sigma_t^2}{\sigma_s^2 - \sigma_t^2} \right)^2 \right)
	\end{equation}
	is a constant independent of $q$.
\end{corollary}

\begin{corollary}[Special case: $b \geq k + r$] \label{cor:bkr}
	For any choices of $r, s \geq 0$, for $j = 1, \cdots, k$,
	\begin{equation}
	\sigma_j \geq \sigma_j\left( \mathbf{B}_k \right) \geq \frac{\sigma_{j+s}}{\sqrt{1 + \mathcal{C}_{b\geq k+r}^2 T_{2q+1}^{-2} \left( 1 + 2 \cdot \frac{\sigma_j - \sigma_{j+s+r+1}}{\sigma_{j+s+r+1}} \right)}}
	\end{equation}
	where
	\begin{equation}
	\mathcal{C}_{b\geq k+r} = \left\Vert \widetilde{\mathbf{\Omega}}_{41} \right\Vert_2 \left\Vert \widetilde{\mathbf{\Omega}}^{-1}_{11} \right\Vert_2
	\end{equation}
	is a constant independent of both $q$, the iteration parameter, and $\mathbf{\Sigma}$, the spectrum of $\mathbf{A}$.  
\end{corollary}

Choosing optimally the analysis parameters $r, s$, we arrive at a result coinciding asymptotically with the conclusions reached in \cite{musco2015randomized}.

\begin{theorem}
	Let $\mathbf{B}_k$ be the matrix returned by running Alg.~\ref{alg:blk_lanczos} with the block size $b = k$. Assume $\mathbf{\Omega}$ is chosen such that $\widetilde{\mathbf{\Omega}}_{11}$ is nonsingular. Then, for $j = 1, \cdots, k$
	\begin{equation}
	\sigma_j \geq \sigma_j\left(\mathbf{B}_k \right) \geq \sigma_j e^{\mathcal{O}\left( -\frac{\log (\mathcal{A}(4q+2))^2}{(4q+2)^2} \right)}
	\end{equation}
	where $\mathcal{A} = 2\mathcal{C}_{b\geq k+r}$ is a constant independent of $q$. 
\end{theorem}

Finally, from Theorem~\ref{thm:svconvergence} we may derive the following result, which states that for certain matrices with singular spectrum rapidly decaying to $0$, the RBL algorithm converges superlinearly.

\begin{theorem}
	\label{thm:superlinear}
	Assume the singular value spectrum of $\mathbf{A}$ decays such that $\sigma_i \rightarrow 0$. Let $\mathbf{B}_k$ be the rank $k$ approximation of $\mathbf{A}$ returned by Alg.~\ref{alg:blk_lanczos}. Assume additionally that the hypothesis and notation of Theorem~\ref{thm:svconvergence} hold. Then
	\begin{equation}
	\sigma_j (\mathbf{B}_k) \rightarrow \sigma_j
	\end{equation}
	superlinearly in $q$, the number of iterations. 
\end{theorem}

This theorem validates long observed empirical behaviors of block Lanczos algorithms. In Section~\ref{sec:numerical_experiments}, we show two examples of typical data matrices with spectrums that fall under this regime, and the expected superlinear convergence behavior.

\subsection{Intuition}
Our analysis makes use of the following three ideas:

\begin{figure}
	\centering
	\caption{Chebyshev polynomials $T_n(x)$ grow much faster than monomials of the same degree $M_n(x) = x^n$ in the interval $\vert x \vert > 1$.}
	\label{fig:chebyshev}
	\includegraphics[scale=0.6]{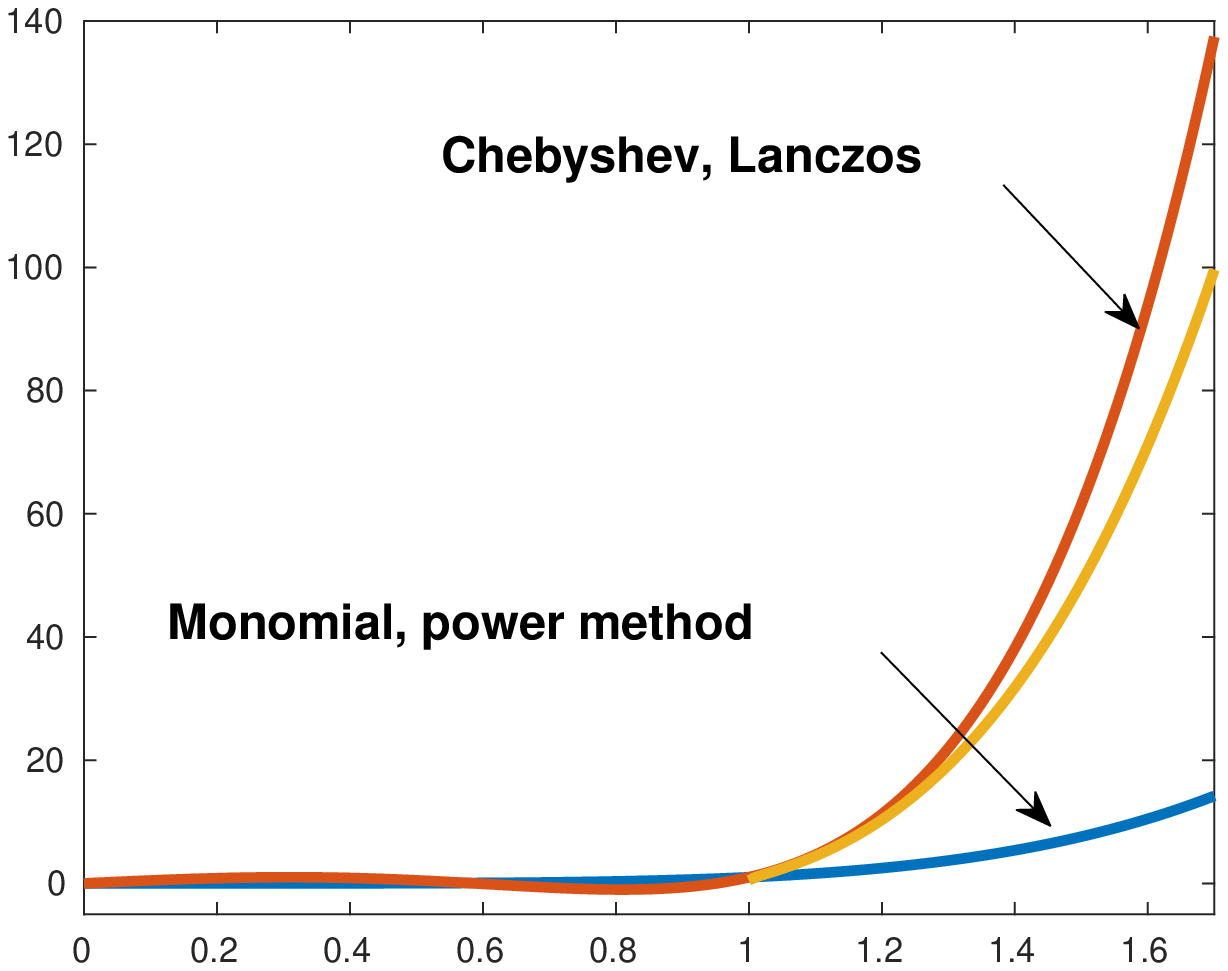}
\end{figure}

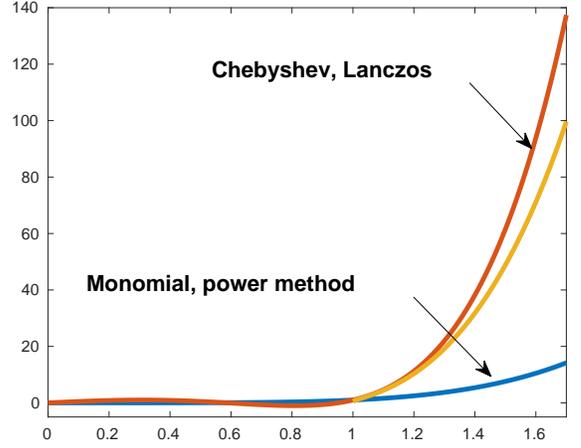
\begin{figure}
	\centering
	\caption{Auxiliary analysis parameters $r, s$ are adjusted to create a sufficient singular spectrum ``gap'' to drive convergence.}
	\label{fig:params}
	\begin{tikzpicture}[scale=0.7]
	\draw [thick, -] (0,0) -- (12,0);
	\draw [thick, -] (1.44, -0.2) -- (1.44, 0.2);
	\node[align=center] at (1.44, -0.5) {0};
	\draw [thick, -] (2.88, -0.2) -- (2.88, 0.2);
	\node[align=center] at (2.88, -0.5) {$\sigma_n$};
	\draw [thick, -] (6.00, -0.2) -- (6.00, 0.2);
	\node[align=center] at (6.00, -0.5) {$\sigma_{k+s+r+1}$};
	\draw [thick, -] (8.40, -0.2) -- (8.40, 0.2);
	\node[align=center] at (8.40, +0.5) {$\sigma_{k+s}$};
	\draw [thick, -] (8.64, -0.2) -- (8.64, 0.2);
	\node[align=center] at (8.64, -0.5) {$\sigma_k$};
	\draw [thick, -] (10.08, -0.2) -- (10.08, 0.2);
	\node[align=center] at (10.08, -0.5) {$\sigma_1$};
	\draw [thin, -] (8.60, -0.1) -- (8.60, 0.1);
	\draw [thin, -] (8.55, -0.1) -- (8.55, 0.1);
	\draw [thin, -] (8.50, -0.1) -- (8.50, 0.1);
	\draw [thin, -] (8.45, -0.1) -- (8.45, 0.1);
	\draw[decoration={brace,mirror,raise=5pt},decorate]	(6.00,-0.7) -- (8.64,-0.7);
	\node[align=center] at (7.44, -1.5) {``gap''};
	\end{tikzpicture}
\end{figure}

\begin{enumerate}
	\item the growth behavior of Chebyshev polynomials, a traditional ingredient in the analysis of Lanczos iteration methods, (Fig.~\ref{fig:chebyshev})
	\item the choice of a clever orthonormal basis for the Krylov subspace, an idea adapted from \cite{gu2015subspace},
	\item the creation of a spectrum ``gap'', by separating the spectrum of $\mathbf{A}$ into those singular values that are ``close'' to $\sigma_k$, and those that are sufficiently smaller in magnitude, using auxiliary analysis parameters $r,s$. (Fig.~\ref{fig:params})
\end{enumerate}

\subsection{Analysis} \label{ssec:analysis}

We are interested in the column span of the Krylov subspace matrix $\mathbf{K}$. Let the singular value decomposition of $\mathbf{A}$ be denoted as $\mathbf{A} = \mathbf{U} \mathbf{\Sigma} \mathbf{V}^T$. Then, we may write 
\begin{align}
\mathbf{K} &= \begin{bmatrix} \mathbf{A} \mathbf{\Omega} & (\mathbf{A} \mathbf{A}^T) \mathbf{A} \mathbf{\Omega} & \cdots & (\mathbf{A} \mathbf{A}^T)^q \mathbf{A} \mathbf{\Omega} \end{bmatrix} \nonumber \\
&= \begin{bmatrix} \mathbf{U} \mathbf{\Sigma} \mathbf{V}^T \mathbf{\Omega} & \mathbf{U} \mathbf{\Sigma}^{2+1} \mathbf{V}^T \mathbf{\Omega} & \cdots & \mathbf{U} \mathbf{\Sigma}^{2q+1} \mathbf{V}^T \mathbf{\Omega} \end{bmatrix} \nonumber \\
&= \mathbf{U} \mathbf{\Sigma} \begin{bmatrix} \widehat{\mathbf{\Omega}} & \widehat{\mathbf{\Sigma}} \widehat{\mathbf{\Omega}} & \cdots & \widehat{\mathbf{\Sigma}}^q \widehat{\mathbf{\Omega}} \end{bmatrix}
\end{align}
where for notational convenience we have defined the quantities $\widehat{\mathbf{\Omega}} \equiv \mathbf{V}^T \mathbf{\Omega}$ and $\widehat{\mathbf{\Sigma}} \equiv \mathbf{\Sigma}^2$.

We ``factor out'' the component of the Krylov subspace that drives convergence from the component that is related to the initial starting subspace but independent of $q$. To this end, define for $0 \leq p \leq q$,
\begin{equation}
\mathbf{K}_p \equiv \mathbf{U} T_{2p+1} ( \mathbf{\Sigma} ) \begin{bmatrix} \widehat{\mathbf{\Omega}} & \widehat{\mathbf{\Sigma}} \widehat{\mathbf{\Omega}} & \cdots & \widehat{\mathbf{\Sigma}}^{q-p} \widehat{\mathbf{\Omega}} \end{bmatrix}  \label{eqn:kp}
\end{equation} 

The matrices $\mathbf{K}$ and $\widehat{\mathbf{K}}$ are related as 
\begin{equation}
\mathrm{span}\left\{ \mathbf{K}_p \right\} \subseteq \mathrm{span} \left\{ \mathbf{K} \right\}
\end{equation}
In light of this, since Step 3 of Alg.~\ref{alg:blk_lanczos} is a projection, we are justified in our analysis to work with $\mathbf{K}_p$ instead of the more complicated $\mathbf{K}$. 

Next, we multiply $\mathbf{K}_p$ by a specially constructed, full rank matrix $\mathbf{X}$. This operation will preserve the subspace spanned by the columns of $\mathbf{K}_p$, but align, as much as possible, the first $k$ columns to the direction of the leading $k$ singular vectors. 

For all $0 \leq p \leq q$, let
\begin{equation}
\mathbf{V}_p \equiv \begin{bmatrix} \widehat{\mathbf{\Omega}} & \widehat{\mathbf{\Sigma}} \widehat{\mathbf{\Omega}} & \cdots & \widehat{\mathbf{\Sigma}}^{q-p} \widehat{\mathbf{\Omega}} \end{bmatrix}
\end{equation}
denote the generalized Vandermonde matrix from Eqn.~\ref{eqn:kp} and partition this matrix as follows:
\begin{equation}
\mathbf{V}_p = \begin{bmatrix} \mathbf{V}_{11} & \mathbf{V}_{12} \\ \mathbf{V}_{21} & \mathbf{V}_{22} \\ \mathbf{V}_{31} & \mathbf{V}_{32} \\ \mathbf{V}_{41} & \mathbf{V}_{42} \end{bmatrix} \label{eqn:vp}
\end{equation}
where the blocks in the first dimension are sized $k, s, r, t = n - (k+s+r)$ and the blocks in the second dimension are sized $k, r$. Intuitively, $s$ is used to handle duplicate or clustered singular values, while $r$ is used to create the ``gap'' that drives convergence (Fig.~\ref{fig:params}). With this partition, we examine the convergence behavior viewed as an accentuation of the ``gap'' by the appropriate Chebysehv polynomial. 

We show the existence of a(t least one) special non-singular $\mathbf{X} \in \mathbb{R}^{(k+r) \times (k+r)}$ such that 
\begin{align}
\mathbf{K}_p \mathbf{X} &= \mathbf{U} T_{2p+1}(\mathbf{\Sigma}) \mathbf{V}_p \mathbf{X} \\
&= \mathbf{U} \begin{bmatrix} \mathbf{Q}_{11} & \widehat{\mathbf{V}}_{12} \\ \mathbf{Q}_{21} & \widehat{\mathbf{V}}_{22} \\ \mathbf{0} & \widehat{\mathbf{V}}_{32} \\ \mathbf{H} & \widehat{\mathbf{V}}_{42} \end{bmatrix} \label{eqn:gapmatrix}
\end{align}
with $\begin{bmatrix} \mathbf{Q}_{11} \\ \mathbf{Q}_{21} \end{bmatrix}$ a column orthogonal matrix. Notice the ``gap'' in the $(3,1)$ block of size $r$ is created by using $\mathbf{X}$ to align the columns of $\mathbf{K}_p$.

We explicit construct such an $\mathbf{X}$. Partition
\begin{align}
\mathbf{X} &= \begin{bmatrix} \mathbf{X}_{11} & \mathbf{X}_{12} \\ \mathbf{X}_{21} & \mathbf{X}_{22} \end{bmatrix} \label{eqn:xblocked} \\
\mathbf{\Sigma} &= \begin{bmatrix} \mathbf{\Sigma}_1 & & & \\ & \mathbf{\Sigma}_2 & & \\ & & \mathbf{\Sigma}_3 & \\ & & & \mathbf{\Sigma}_4 \end{bmatrix}
\end{align}
where each dimension of $\mathbf{X}$ is sized $k, r$, and each dimension of $\mathbf{\Sigma}$ is sized $k, s, r, t=n-(k+s+r)$. Then,
\begin{align}
T_{2p+1}(\mathbf{\Sigma}) \mathbf{V}_p \mathbf{X} \equiv \left[ \begin{array}{c|c} \begin{pmatrix} \widehat{\mathbf{V}}_{11} \\ \widehat{\mathbf{V}}_{21} \end{pmatrix} & \cdots \\ \hline \widehat{\mathbf{V}}_{31} & \cdots \\ \hline \widehat{\mathbf{V}}_{41} & \cdots \end{array} \right]
\end{align} 
where 
\begin{align*}
\begin{pmatrix} \widehat{\mathbf{V}}_{11} \\ \widehat{\mathbf{V}}_{21} \end{pmatrix} &= \begin{pmatrix}  T_{2p+1}(\mathbf{\Sigma}_1) & \\ &  T_{2p+1}(\mathbf{\Sigma}_2) \end{pmatrix} \begin{pmatrix} \mathbf{V}_{11} & \mathbf{V}_{12} \\ \mathbf{V}_{21} & \mathbf{V}_{22} \end{pmatrix} \begin{pmatrix} \mathbf{X}_{11} \\ \mathbf{X}_{21} \end{pmatrix} \\
\widehat{\mathbf{V}}_{31} &= T_{2p+1}(\mathbf{\Sigma}_3) (\mathbf{V}_{31} \mathbf{X}_{11} + \mathbf{V}_{32} \mathbf{X}_{21}) \\
\widehat{\mathbf{V}}_{41} &=  T_{2p+1}(\mathbf{\Sigma}_4) (\mathbf{V}_{41} \mathbf{X}_{11} + \mathbf{V}_{42} \mathbf{X}_{21})
\end{align*}

Setting 
\begin{equation}
\mathbf{X}_{21} = - \mathbf{V}_{32}^{-1} \mathbf{V}_{31} \mathbf{X}_{11} \label{eqn:x21def}
\end{equation}
ensures the $(2,1)$ block of dimensions $r \times k$ to be $\widehat{\mathbf{V}}_{31} = \mathbf{0}$, and causes the $(1,1)$ block of dimensions $(k+s) \times k$ to become
\begin{align*}
\begin{pmatrix} \widehat{\mathbf{V}}_{11} \\ \widehat{\mathbf{V}}_{21} \end{pmatrix} = \begin{bmatrix} T_{2p+1}(\mathbf{\Sigma}_1) & \\ &  T_{2p+1}(\mathbf{\Sigma}_2) \end{bmatrix} \begin{bmatrix}  \mathbf{V}_{11} - \mathbf{V}_{12} \mathbf{V}_{32}^{-1} \mathbf{V}_{31} \\ \mathbf{V}_{21} - \mathbf{V}_{22} \mathbf{V}_{32}^{-1} \mathbf{V}_{31} \end{bmatrix} \mathbf{X}_{11}
\end{align*}

We can then take the QR factorization
\begin{equation}
\widetilde{\mathbf{Q}}\widetilde{\mathbf{R}} = \begin{bmatrix}  T_{2p+1}(\mathbf{\Sigma}_1) & \\ & T_{2p+1}(\mathbf{\Sigma}_2) \end{bmatrix} \begin{bmatrix}  \mathbf{V}_{11} - \mathbf{V}_{12} \mathbf{V}_{32}^{-1} \mathbf{V}_{31} \\ \mathbf{V}_{21} - \mathbf{V}_{22} \mathbf{V}_{32}^{-1} \mathbf{V}_{31} \end{bmatrix} \label{eqn:rtildedef}
\end{equation}
and set 
\begin{equation}
\mathbf{X}_{11} = \widetilde{\mathbf{R}}^{-1} \label{eqn:x11def}
\end{equation}

This ensures that
\begin{equation}
\begin{pmatrix} \widehat{\mathbf{V}}_{11} \\ \widehat{\mathbf{V}}_{21} \end{pmatrix}^T \begin{pmatrix} \widehat{\mathbf{V}}_{11} \\ \widehat{\mathbf{V}}_{21} \end{pmatrix} = \left( \widetilde{\mathbf{Q}}\widetilde{\mathbf{R}}\widetilde{\mathbf{R}}^{-1} \right)^T \left( \widetilde{\mathbf{Q}}\widetilde{\mathbf{R}}\widetilde{\mathbf{R}}^{-1} \right) = \mathbf{I}
\end{equation}

Let Eqn.~(\ref{eqn:x11def}) and Eqn.~(\ref{eqn:x21def}) define $\mathbf{X}_{11}$ and $\mathbf{X}_{21}$ respectively as 
\begin{equation}
\begin{bmatrix} \mathbf{X}_{11} \\ \mathbf{X}_{21} \end{bmatrix} = \begin{bmatrix} \mathbf{I} \\ -\mathbf{V}_{32}^{-1} \mathbf{V}_{31} \end{bmatrix} \widetilde{\mathbf{R}}^{-1} \label{eqn:xdefcol2}
\end{equation}
We specify 
\begin{equation}
\begin{bmatrix} \mathbf{X}_{12} \\ \mathbf{X}_{22} \end{bmatrix} \equiv \begin{bmatrix} \mathbf{X}_{11} \\ \mathbf{X}_{21} \end{bmatrix}^{\perp} \label{eqn:xdefcol1}
\end{equation}
to provide a complete description of $\mathbf{X}$ which satisfies Eqn.~\ref{eqn:gapmatrix}.

\begin{remark}
	In order for the above derivation and thus Eqn.~(\ref{eqn:xdefcol1}) and Eqn.~(\ref{eqn:xdefcol2}) to be valid, the following conditions must hold: $\mathbf{\Omega}$ is chosen to allow 
	\begin{itemize}
		\item $\mathbf{V}_{32}$ to be non-singular and thus invertible,
		\item $\mathbf{V}_{11}-\mathbf{V}_{12} \mathbf{V}_{32}^{-1} \mathbf{V}_{31}$ to be non-singular and thus $\widetilde{\mathbf{R}}$ to be invertible. Note that this expression is the Schur complement of the $(k+r) \times (k+r)$ matrix $\begin{bmatrix} \mathbf{V}_{11} & \mathbf{V}_{12} \\ \mathbf{V}_{31} & \mathbf{V}_{32} \end{bmatrix}$ with respect to the $\mathbf{V}_{32}$ block. 
	\end{itemize}
	\label{lst:conditions}
\end{remark}

We present a first result on a lower bound for the singular value of $\mathbf{B}_k$. 
\begin{lemma}
	\label{lem:svconvlemma}
	Let $\mathbf{B}_k$ be the matrix returned by Alg.~\ref{alg:blk_lanczos}, let $\mathbf{H}$ be as defined in Eqn.~(\ref{eqn:gapmatrix}), and assume that the two conditions in Remark~\ref{lst:conditions} hold. Then,
	\begin{equation}
	\sigma_k(\mathbf{B}_k) \geq \frac{\sigma_{k+s}}{\sqrt{1 + \Vert \mathbf{H} \Vert_2^2}}
	\end{equation}
\end{lemma}
\begin{proof}
	The matrix returned by Alg.~\ref{alg:blk_lanczos} is the $k$-truncated SVD of $\mathbf{Q}\mathbf{Q}^T \mathbf{A}$, where the columns of $\mathbf{Q}$ are an orthonormal basis for the column span of $\mathbf{K}$. By construction, it follows that
	\begin{equation}
	\sigma_k(\mathbf{B}_k) \geq \sigma_k \left( \widehat{\mathbf{Q}}_p \widehat{\mathbf{Q}}_p^T \mathbf{A} \right) \label{eqn:sveqn1}
	\end{equation}
	where $\widehat{\mathbf{Q}}_p$ contains columns that form an orthonormal basis for the column span of $\mathbf{K}_p \mathbf{X}$. 
	
	In particular, let $\widehat{\mathbf{Q}}_p \widehat{\mathbf{R}}_p$ be the QR factorization of $\mathbf{K}_p \mathbf{X}$, partitioned as follows:
	\begin{equation}
	\mathbf{K}_p \mathbf{X} = \widehat{\mathbf{Q}}_p \widehat{\mathbf{R}}_p = \begin{bmatrix} \widehat{\mathbf{Q}}_1 & \widehat{\mathbf{Q}}_2 \end{bmatrix} \begin{bmatrix} \widehat{\mathbf{R}}_{11} & \widehat{\mathbf{R}}_{12} \\ & \widehat{\mathbf{R}}_{22} \end{bmatrix} \label{eqn:kpxqrfact}
	\end{equation}
	where the block dimensions are sized $k,s$, as appropriate. 
	
	We can then write
	\begin{align*}
	\widehat{\mathbf{Q}}_p &\widehat{\mathbf{Q}}_p^T \mathbf{A} \\
	&= \widehat{\mathbf{Q}}_p \begin{bmatrix} \widehat{\mathbf{Q}}_1^T \\ \widehat{\mathbf{Q}}_2^T \end{bmatrix} \mathbf{U} \left[ \begin{array}{c|c} \begin{pmatrix} \mathbf{\Sigma}_1 & \\ & \mathbf{\Sigma}_2 \\ \mathbf{0} & \mathbf{0} \\ \mathbf{0} & \mathbf{0} \end{pmatrix} & \begin{pmatrix} \mathbf{0} & \mathbf{0} \\ \mathbf{0} & \mathbf{0} \\ \mathbf{\Sigma}_3 & \\ & \mathbf{\Sigma}_4 \end{pmatrix} \end{array} \right] \mathbf{V}^T \\
	&= \widehat{\mathbf{Q}}_p \left[ \begin{array}{c|c} \widehat{\mathbf{Q}}_1^T \mathbf{U} \begin{pmatrix} \mathbf{\Sigma}_1 & \\ & \mathbf{\Sigma}_2 \\ \mathbf{0} & \mathbf{0} \\ \mathbf{0} & \mathbf{0} \end{pmatrix} & \widehat{\mathbf{Q}}_{1}^T \mathbf{U} \begin{pmatrix} \mathbf{0} & \mathbf{0} \\ \mathbf{0} & \mathbf{0} \\ \mathbf{\Sigma}_3 & \\ & \mathbf{\Sigma}_4 \end{pmatrix} \\ \hline \widehat{\mathbf{Q}}_2^T \mathbf{U} \begin{pmatrix} \mathbf{\Sigma}_1 & \\ & \mathbf{\Sigma}_2 \\ \mathbf{0} & \mathbf{0} \\ \mathbf{0} & \mathbf{0} \end{pmatrix} & \widehat{\mathbf{Q}}_2^T \mathbf{U} \begin{pmatrix} \mathbf{0} & \mathbf{0} \\ \mathbf{0} & \mathbf{0} \\ \mathbf{\Sigma}_3 & \\ & \mathbf{\Sigma}_4 \end{pmatrix} \end{array} \right] \mathbf{V}^T
	\end{align*}
	
	By the Cauchy interlacing theorem for singular values, it follows that 
	\begin{equation}
	\sigma_k \left( \widehat{\mathbf{Q}}_p \widehat{\mathbf{Q}}_p^T \mathbf{A} \right) \geq \sigma_k \left( \widehat{\mathbf{Q}}_1^T \mathbf{U} \begin{pmatrix} \mathbf{\Sigma}_1 & \\ & \mathbf{\Sigma}_2 \\ \mathbf{0} & \mathbf{0} \\ \mathbf{0} & \mathbf{0} \end{pmatrix} \right) \label{eqn:sveqn2}
	\end{equation}
	
	We can compare the first $k$ columns of Eqn.~(\ref{eqn:kpxqrfact}) with the expression in Eqn.~(\ref{eqn:gapmatrix}) to see that
	\begin{equation}
	\widehat{\mathbf{Q}}_1 \widehat{\mathbf{R}}_{11} = \mathbf{U} \begin{bmatrix} \mathbf{Q}_{11} \\ \mathbf{Q}_{21} \\ \mathbf{0} \\ \mathbf{H} \end{bmatrix} \label{eqn:r11def}
	\end{equation}
	which helps us to write
	\begin{align}
	\widehat{\mathbf{Q}}_1^T \mathbf{U} \begin{pmatrix} \mathbf{\Sigma}_1 & \\ & \mathbf{\Sigma}_2 \\ \mathbf{0} & \mathbf{0} \\ \mathbf{0} & \mathbf{0} \end{pmatrix} &= \left( \mathbf{U} \begin{pmatrix} \mathbf{Q}_{11} \\ \mathbf{Q}_{21} \\ \mathbf{0} \\ \mathbf{H} \end{pmatrix} \widehat{\mathbf{R}}_{11}^{-1} \right)^T \mathbf{U} \begin{pmatrix} \mathbf{\Sigma}_1 & \\ & \mathbf{\Sigma}_2 \\ \mathbf{0} & \mathbf{0} \\ \mathbf{0} & \mathbf{0} \end{pmatrix} \nonumber \\
	&= \widehat{\mathbf{R}}_{11}^{-T} \begin{pmatrix} \mathbf{Q}_{11} \\ \mathbf{Q}_{21} \\ \mathbf{0} \\ \mathbf{H} \end{pmatrix}^T \begin{pmatrix} \mathbf{\Sigma}_1 & \\ & \mathbf{\Sigma}_2 \\ \mathbf{0} & \mathbf{0} \\ \mathbf{0} & \mathbf{0} \end{pmatrix} \nonumber \\
	&= \widehat{\mathbf{R}}_{11}^{-T} \begin{bmatrix} \mathbf{Q}_{11}^T \mathbf{\Sigma}_1 & \mathbf{Q}_{21}^T \mathbf{\Sigma}_2 \end{bmatrix} \label{eqn:sveqn3}
	\end{align}
	
	On the other hand, 
	\begin{align}
	\sigma_{k+s} &= \sigma_k \left( \sigma_{k+s} \begin{bmatrix} \mathbf{Q}_{11}^T & \mathbf{Q}_{21}^T \end{bmatrix} \right) \nonumber \\
	&\leq \sigma_k \left( \begin{bmatrix} \mathbf{Q}_{11}^T \mathbf{\Sigma}_1 & \mathbf{Q}_{21}^T \mathbf{\Sigma}_2 \end{bmatrix} \right) \nonumber \\
	&= \sigma_k \left( \widehat{\mathbf{R}}_{11}^T \widehat{\mathbf{R}}_{11}^{-T} \begin{bmatrix} \mathbf{Q}_{11}^T \mathbf{\Sigma}_1 & \mathbf{Q}_{21}^T \mathbf{\Sigma}_2 \end{bmatrix} \right) \nonumber \\
	&\leq \Vert \widehat{\mathbf{R}}_{11}^T \Vert_2 \, \sigma_k \left( \widehat{\mathbf{R}}_{11}^{-T} \begin{bmatrix} \mathbf{Q}_{11}^T \mathbf{\Sigma}_1 & \mathbf{Q}_{21}^T \mathbf{\Sigma}_2 \end{bmatrix} \right) \label{eqn:sveqn4}
	\end{align}
	
	Combining Eqns.~(\ref{eqn:sveqn1}),~(\ref{eqn:sveqn2}),~(\ref{eqn:sveqn3}), and~(\ref{eqn:sveqn4}), we obtain
	\begin{equation}
	\sigma_k(\mathbf{B}_k) \geq \frac{\sigma_{k+s}}{\Vert \widehat{\mathbf{R}}_{11}^T \Vert_2}
	\end{equation}
	
	With the help of Eqn.~(\ref{eqn:r11def}), 
	\begin{align}
	\widehat{\mathbf{R}}_{11}^T \widehat{\mathbf{R}}_{11} &= \widehat{\mathbf{R}}_{11}^T \left(\mathbf{U}^T \widehat{\mathbf{Q}}_1 \right)^T \left(\mathbf{U}^T \widehat{\mathbf{Q}}_1 \right) \widehat{\mathbf{R}}_{11} \\
	&= \begin{bmatrix} \mathbf{Q}_{11} \\ \mathbf{Q}_{21} \end{bmatrix}^T \begin{bmatrix} \mathbf{Q}_{11} \\ \mathbf{Q}_{21} \end{bmatrix} + \mathbf{H}^T \mathbf{H} \\
	&= \mathbf{I} + \mathbf{H}^T \mathbf{H}
	\end{align}
	, which completes the proof. 
\end{proof}

We are now in a position to provide the proof for Theorem~\ref{thm:svconvergence}
\begin{proof}
	With an eye toward Lemma~\ref{lem:svconvlemma}, we proceed by providing a bound for $\Vert \mathbf{H} \Vert_2^2$. 
	\begin{align*}
	& \, \Vert \mathbf{H} \Vert_2^2 \\
	=& \, \sigma_1^2 \left( \mathbf{H} \mathbf{H}^T \right) \\
	=& \, \sigma_1^2 \Bigg( T_{2p+1}(\mathbf{\Sigma}_4) (\mathbf{V}{41}-\mathbf{V}_{42}\mathbf{V}_{32}^{-1}\mathbf{V}_{31}) \left(\widetilde{\mathbf{R}}^T \widetilde{\mathbf{R}} \right)^{-1} \\
	&(\mathbf{V}{41}-\mathbf{V}_{42}\mathbf{V}_{32}^{-1}\mathbf{V}_{31})^T T_{2p+1}(\mathbf{\Sigma}_4) \Bigg) \\
	=& \, \sigma_1^2 \Bigg( T_{2p+1}(\mathbf{\Sigma}_4) (\mathbf{V}{41}-\mathbf{V}_{42}\mathbf{V}_{32}^{-1}\mathbf{V}_{31}) \nonumber \\ 
	&\Bigg( \begin{bmatrix} \mathbf{V}_{11} - \mathbf{V}_{12} \mathbf{V}_{32}^{-1} \mathbf{V}_{31} \\ \mathbf{V}_{21}-\mathbf{V}_{22}\mathbf{V}_{32}^{-1} \mathbf{V}_{31} \end{bmatrix}^T \begin{bmatrix} T_{2p+1}^2(\mathbf{\Sigma}_1) & \\ & T_{2p+1}^2(\mathbf{\Sigma}_2) \end{bmatrix} \\
	&\hspace{4cm}\begin{bmatrix} \mathbf{V}_{11} - \mathbf{V}_{12} \mathbf{V}_{32}^{-1} \mathbf{V}_{31} \\ \mathbf{V}_{21}-\mathbf{V}_{22}\mathbf{V}_{32}^{-1} \mathbf{V}_{31} \end{bmatrix} \Bigg)^{-1} \nonumber \\
	&(\mathbf{V}{41}-\mathbf{V}_{42}\mathbf{V}_{32}^{-1}\mathbf{V}_{31})^T T_{2p+1}(\mathbf{\Sigma}_4) \Bigg) \\
	\leq& \, \sigma_1^2 \Bigg( T_{2p+1}(\mathbf{\Sigma}_4) (\mathbf{V}{41}-\mathbf{V}_{42}\mathbf{V}_{32}^{-1}\mathbf{V}_{31}) \nonumber \\ 
	( &(\mathbf{V}_{11}-\mathbf{V}_{12}\mathbf{V}_{32}^{-1}\mathbf{V}_{31})^T T_{2p+1}^2(\mathbf{\Sigma}_1) (\mathbf{V}_{11}-\mathbf{V}_{12}\mathbf{V}_{32}^{-1}\mathbf{V}_{31}) )^{-1} \nonumber \\
	&(\mathbf{V}{41}-\mathbf{V}_{42}\mathbf{V}_{32}^{-1}\mathbf{V}_{31})^T T_{2p+1}(\mathbf{\Sigma}_4) \Bigg) \\
	=& \, \Vert T_{2p+1}(\mathbf{\Sigma}_4) (\mathbf{V}_{41}-\mathbf{V}_{42}\mathbf{V}_{32}^{-1} \mathbf{V}_{31}) \\
	&(\mathbf{V}_{11} -\mathbf{V}_{12}\mathbf{V}_{32}^{-1}\mathbf{V}_{31})^{-1} T_{2p+1}^{-1}(\mathbf{\Sigma}_1) \Vert_2^2 \\
	\leq& \, T_{2p+1}^{-2}\left( 1 + 2 \cdot \frac{\sigma_k-\sigma_{k+s+r+1}}{\sigma_{k+s+r+1}} \right) \\
	&\Vert (\mathbf{V}_{41}-\mathbf{V}_{42}\mathbf{V}_{32}^{-1} \mathbf{V}_{31}) (\mathbf{V}_{11} -\mathbf{V}_{12}\mathbf{V}_{32}^{-1}\mathbf{V}_{31})^{-1} \Vert_2^2
	\end{align*}
	The $1 + 2 \cdot \frac{\sigma_k - \sigma_{k+s+r+1}}{\sigma_{k+s+r+1}}$ factor is interpreted as shifting the Chebyshev polynomial $T_{2p+1}$ onto the interval $[0, \sigma_{k+s+r+1}]$, so that the tail of the singular spectrum is bounded by $1$ and convergence is driven by the growth of the Chebyshev polynomial on the $[\sigma_k, \cdots, \sigma_1]$ part of the spectrum that we are interested in.
	
	Repeating the previous argument for $1 \leq j \leq k$ completes the proof for the bound on $\sigma_j(\mathbf{B}_k)$.
\end{proof}

Due to space constraint, we omit the proofs for the corollaries of Theorem~\ref{thm:svconvergence}. They are similar in flavor to the proof above and involve constructions of specifically chosen $\mathbf{X}$ matrices in each case.

We close by providing the proof for Theorem~\ref{thm:superlinear}.

\begin{proof}
	The statement of the theorem is equivalent to the statement that 
	\begin{equation}
	\mathcal{C} \, T_{2p+1}^{-1}\left( 1+2 \cdot \frac{\sigma_j - \sigma_{j+r+1}}{\sigma_{j+r+1}} \right) \rightarrow 0
	\end{equation}
	superlinearly. For notational convenience we assume $\sigma_j$ is not a multiple singular value and we have chosen $s=0$; otherwise, the following argument can be made for the largest choice of $s$ such that $\sigma_{j+s} = \sigma_j$. 
	
	Recall that a sequence $\left\{ a_n \right\}$ converges superlinearly to $a$ if 
	\begin{equation}
	\lim_{n \rightarrow \infty} \frac{\vert a_{n+1} - a \vert}{\vert a_n - a \vert} = 0
	\end{equation}
	
	For any fixed $j = 1, \cdots, k$, define 
	\begin{align*}
	a_q &\equiv \mathcal{C}(r) T^{-1}_{2\left(q+1 - \frac{k+r}{b} \right) + 1} \left(1 + 2 \cdot \frac{\sigma_j - \sigma_{j+r+1}}{\sigma_{j+r+1}} \right) 
	\end{align*}
	where we have explicitly specified the dependence of the constant $\mathcal{C}$ on the analysis parameter $r$, and expressed $p$ in terms of $q$. We approximate 
	\begin{align*}
	a_q &\approx \mathcal{C}(r) \cdot \frac{1}{2} \left( 1 + g + \sqrt{2g} \right)^{-\left( 2\left(q+1-\frac{k+r}{b} \right) + 1\right)} \\
	&\approx \frac{1}{2} \cdot \mathcal{C}(r) \cdot \left(1 + g + \sqrt{2g}\right)^{-2\left( 1-\frac{k+r}{b} \right)+1} \cdot \left(1+g+\sqrt{2g}\right)^{-2q} \\
	&\text{where } \, g = 2 \cdot \frac{\sigma_j - \sigma_{j+r+1}}{\sigma_{j+r+1}} = 2 \cdot \left( \frac{\sigma_j}{\sigma_{j+r+1}} - 1 \right)
	\end{align*}
	Then we argue that $a_{q+1}/a_q \rightarrow 0$ as follows.
	\begin{align}
	\frac{a_{q+1}}{a_q} &= \frac{1}{\left( 1+g+\sqrt{2g} \right)^2} \leq \frac{1}{(1+g)^2} 
	\end{align}
	Since we assume a spectrum such that $\sigma_i \rightarrow 0$ eventually, it is possible to chose $r$ sufficiently large such that $1/(1+g)^2$ is arbitrarily small. 
\end{proof}

Rigorously, the above argument applies only to infinite dimensional operators, as in the finite dimensional case, $r \leq n$ cannot be chosen to be arbitrarily large. However, numerous previous works have noted that in practice, the convergence does tend to exhibit superlinear behavior for certain types of spectrums \cite{saad1994theoretical}. 

\section{Numerical Experiments} \label{sec:numerical_experiments}

\subsection{Computational Complexity}

We will give an arithmetic complexity accounting of the randomized block Lanczos algorithm. The initialization of the random starting matrix $\mathbf{\Omega}$ takes $\mathcal{O}(nb)$ floating-point operations (flops). In step 1, the formation of the Krylov matrix $\mathbf{K}$ consists of $1$ matrix multiplications of $\mathbf{A}\mathbf{\Omega}$ along with $2(q-1)$ accumulated applications of either $\mathbf{A}$ or $\mathbf{A}^T$ for a total of $\mathcal{O}(mnbq)$ flops. The orthonormal basis $\mathbf{Q}$ of $\mathbf{K}$ can be computed using a QR factorization using the standard Householder implementation, which has complexity $\mathcal{O}(m(bq)^2)$. Finally, steps 3 and 4 consists of first forming $\mathbf{Q}^TA$ for $\mathcal{O}(mnbq)$ flops, then computing its truncated SVD factorization. Because the size of this matrix is $qb \times n$ and we expect $qb \approx k$ to be small, we assume its SVD computation is performed with a non-specialized dense matrix algorithm, using $\mathcal{O}(n(bq)^2)$ flops. The final step of projecting the right $k$ singular vectors onto $Q$ is an additional $\mathcal{O}(m(bq)^2)$ flops. 

Overall, the computational complexity of Algorithm~\ref{alg:blk_lanczos} is $\mathcal{O}(mnbq + (m+n)(bq)^2)$. The first term dominates the computations and is the result of performing the matrix multiplications for the computation of the Lanczos block vectors. Fortunately, matrix multiplication is a highly optimized and highly tuned part of many matrix computation libraries, especially for suitably chosen block sizes. 

We draw attention to the fact that the parameters $b$ and $q$ only appear together as the quantity $bq$ in our computational complexity count. This suggests that we may freely vary $b$, $q$ - as long as they vary inversely and the quantity $bq$ remains constant, the cost for running Algorithm~\ref{alg:blk_lanczos} remains comparable. (In practice, this will only hold true for $b > 1$, due to the efficiency of BLAS2 and BLAS3 operations compared with BLAS1 operations.) Given the comparable computational complexity, and assuming the conditions for the convergence of Algorithm~\ref{alg:blk_lanczos} is met, we need not privilege the block size choice $b = k$. In fact, we show empirically that in many cases, it is advantageous to choose block sizes $b$ strictly smaller than $k$.

\subsection{Activities and Sports Dataset}

The Activities and Sports Dataset is a dataset consisting of motion sensor data for $8$ subjects performing $19$ daily/sports activities, for $5$ minutes, sampled at $25$Hz frequency. This dataset can be found at \cite{altun2010comparative}. 

The matrix associated with this dataset is dense and of dimensions $\mathbf{A} \in \mathbb{R}^{9120 \times 5625}$, where each row is a sample and each entry is a double precision float. Figure~\ref{fig:asd_spec} shows a plot of the first $500$ singular values of $\mathbf{A}$. As is typically for data matrices, this matrix exhibit spectrum decay on the order of $\sigma_j = \frac{1}{j^\tau}$, and our theory suggests that in this case, we should observe superlinear convergence for the RBL algorithm.  

\begin{figure}
	\centering
	\includegraphics[scale=0.6]{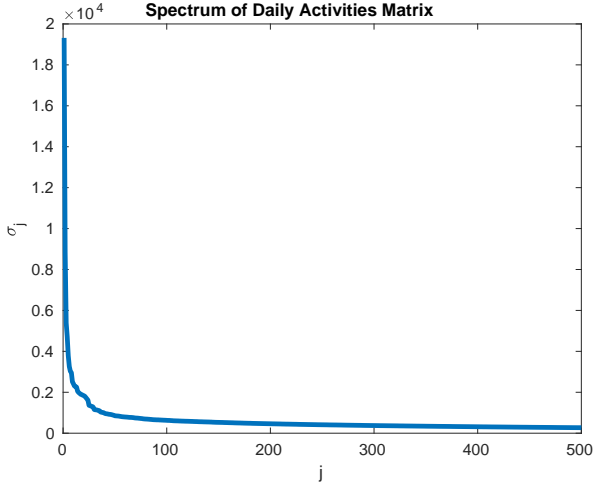}
	\caption{First $500$ singular values of the Daily Activities and Sports Matrix. } 
	\label{fig:asd_spec}
\end{figure}

In this set of experiments, we investigate the convergence of a single singular value with respect to the number of iterations, in addition to the affect of the block size on convergence. We run the RSI and RBL algorithms on the Activities and Sports Dataset matrix with a target rank of $k = 200$, and examine the convergence of $\sigma_1$, $\sigma_{100}$, and $\sigma_{200}$. The results of these experiments are in Figures~\ref{fig:asd_j1},~\ref{fig:asd_j100}, and~\ref{fig:asd_j200}.

\begin{figure}
	\centering
	\includegraphics[scale=0.6]{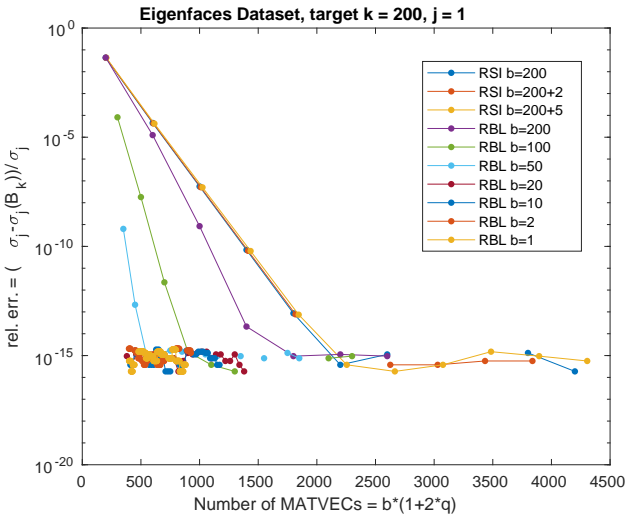}
	\caption{$k=200$ approximation of the Daily Activities Dataset, convergence of $\sigma_{1}$. } 
	\label{fig:asd_j1}
\end{figure}

\begin{figure}
	\centering
	\includegraphics[scale=0.6]{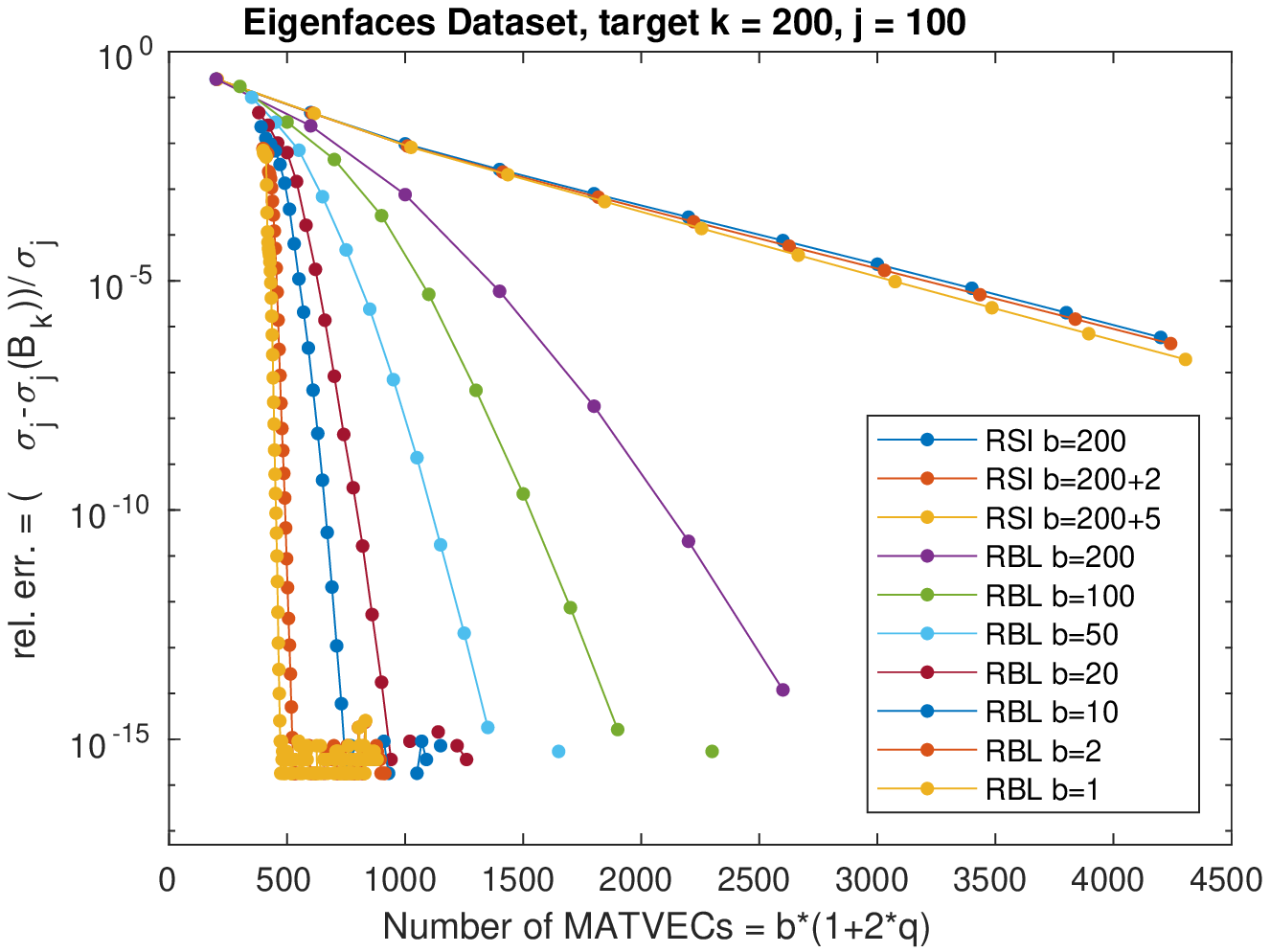}
	\caption{$k=200$ approximation of the Daily Activities Dataset, convergence of $\sigma_{100}$. } 
	\label{fig:asd_j100}
\end{figure}

\begin{figure}
	\centering
	\includegraphics[scale=0.6]{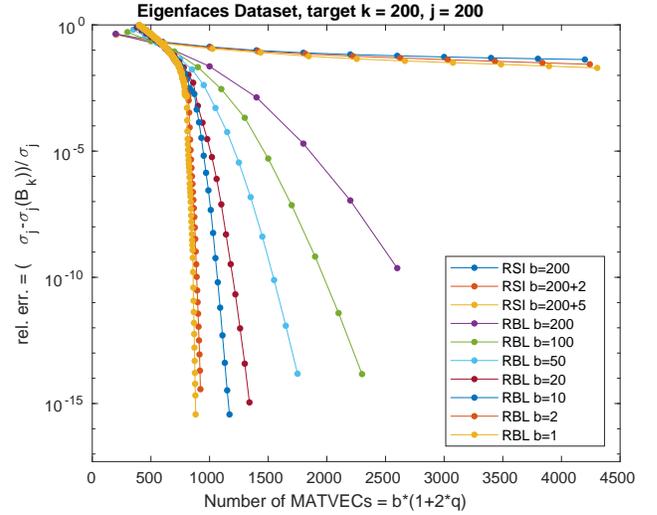}
	\caption{$k=200$ approximation of the Daily Activities Dataset, convergence of $\sigma_{200}$. } 
	\label{fig:asd_j200}
\end{figure}

Each of these plots represent the convergence of a particular singular value. In each plot, each line represents a single parameter setting for the block size $b$, for either the RSI or the RBL algorithm. The $y$-axis is in log scale, and denote 
\begin{equation}
\text{rel. err.} = \frac{\sigma_j - \sigma_j(\mathbf{B}_k)}{\sigma_j}
\end{equation}
, the relative error of the particular singular value we are examining. The $x$-axis is in linear scale, and denote the number of matrix-vector multiplications (MATVECs), a proxy measure for computational complexity. Markers on each line represent successive iterations of the algorithm. In these plots, down and to the left is good - we seek parameter settings that give good convergence for less computational complexity. We observe that, as expected, RSI converges linearly and RBL converges superlinearly. These trends are most clearly seen in Figure~\ref{fig:asd_j200} and is also present in Figure~\ref{fig:asd_j100}. The convergence of $\sigma_1$ is extremely rapid in Figure~\ref{fig:asd_j1}, and reaches double precision in $2$-$5$ iterations for all block sizes. In all cases, for both RBL and RSI, it appears that at the same computational complexity, choosing a smaller block size $b$, leads to more rapid convergence. For example, in Figure~\ref{fig:asd_j200}, we observe that in order for $\sigma_j$ to converge to a relative error of $\sim 10^{-5}$, taking $b = 1$ uses $1/2$ the number of MATVECs as taking $b = k = 200$.

\subsection{Eigenfaces Dataset}
The Eigenfaces dataset is available from the AT\&T Laboratories' Database of Faces \cite{samaria1994parameterisation}, and consists of $10$ different face images of $40$ different subjects at $92 \times 112$ pixels resolution, varying in light, facial expressions, and other details. The widely cited technique for processing this data is via principal component analysis (PCA), where it was observed that each face can be composed in large part from a few prominent ``Eigenfaces'' \cite{turk1991face}.  

The associated matrix is a dense matrix, which is formed by vectorizing each different face image as a column vector. It has dimensions $\mathbf{A} \in \mathbb{R}^{10304 \times 400}$ and is of full numerical rank. The spectrum of this matrix spans $5$ orders of magnitude but decays extremely rapidly, typical of data matrices. In fact, as seen in Figure~\ref{fig:ef_spec}, it drops to zero within the first $50$ largest singular values. 

\begin{figure}
	\centering
	\includegraphics[scale=0.6]{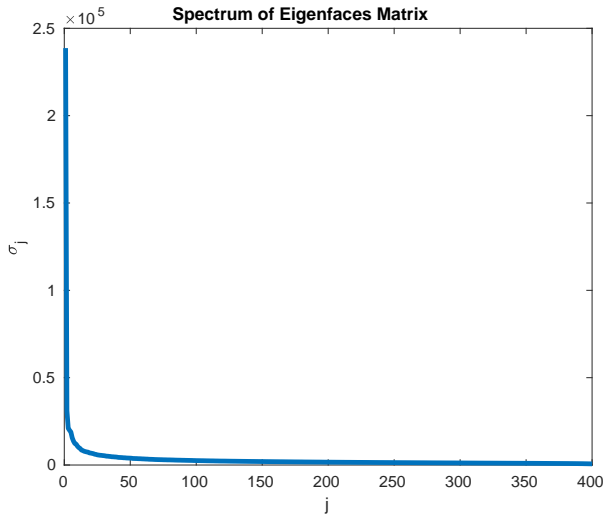}
	\caption{Spectrum of the Eigenfaces Matrix. } 
	\label{fig:ef_spec}
\end{figure}

We repeat the experiments performed in the last section. For this set of experiments, we use the RSI and RBL algorithms to compute rank-$k=100$ approximations for the Eigenfaces matrix, and examine the convergence of $\sigma_{100}$. The result appears in Figure~\ref{fig:ef_j100}. 

\begin{figure}
	\centering
	\includegraphics[scale=0.6]{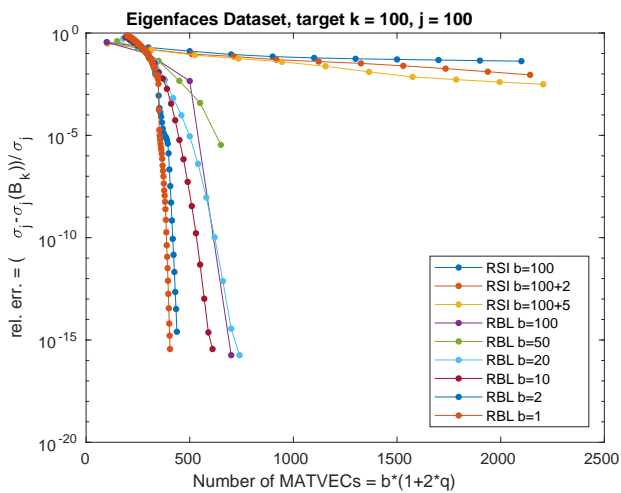}
	\caption{$k=100$ approximation of the Eigenfaces Dataset, convergence of $\sigma_{100}$. } 
	\label{fig:ef_j100}
\end{figure}

We observe similar behavior as those observed for the Daily Activities and Sports Matrix: the RSI algorithm exhibits linear convergence while the RBL algorithm exhibits superlinear convergence; smaller block sizes $b$ appear to converge more quickly for a fixed number of flops.

\section{Conclusions} \label{sec:conclusions}

In this paper, we have derived a novel convergence result for the randomized block Lanczos algorithm. We have shown that for all block sizes, the singular value approximation accuracy for this algorithm converges geometrically in the number of iterations, with a rate that is asymptotically superior to that achieved by the randomized subspace iteration algorithm. We have also shown for a matrix with spectrum decaying to zero, the RBL algorithm converges superlinearly. Additionally, we have provided numerical results in support of our analysis. 

The current work is largely theoretical in nature, and there continues to be need for quality implementations of the Randomized Block Lanczos algorithm to aid its wider adoptability. To this end, continuations of the current work might include such an (possibly parallelized) implementation, along with further investigations of practical choices for the block size parameter $b$ which balances the evident preference for a smaller $b$ for convergence with the advantages of a larger $b$ for computational efficiency and numerical stability.

\bibliographystyle{IEEEtran}

\end{document}